\documentclass[reqno,dvipsnames]{amsart}
\pdfoutput=1

\usepackage{amssymb, amsmath}
\usepackage{mathrsfs}  
\usepackage{mathabx}  
\usepackage[english]{babel}
\usepackage{microtype}               
\usepackage[T1]{fontenc}
\usepackage{lmodern}      
\usepackage[applemac]{inputenc}		

\usepackage{hyperref}
\hypersetup{
	colorlinks   = true,          
	urlcolor     = blue,          
	linkcolor    = blue,          
	citecolor   = red             
}

\usepackage{enumitem}
\setenumerate[1]{label=(\roman*)}    

\usepackage[normalem]{ulem} 
\usepackage{graphicx}
\usepackage[all,cmtip]{xy}
\usepackage{tikz}
\usepackage{tikz-cd}
\usetikzlibrary{matrix,patterns,knots,calc,arrows,shapes,decorations.pathreplacing}
\usepackage{verbatim}

\usepackage{epigraph}   

\newcommand{\norm}[1]{\lVert #1 \rVert}

\newcommand\R{{\mathbb R}}
\newcommand\C{{\mathbb C}}

\DeclareMathOperator{\Link}{L}

\newtheorem{theorem}{Theorem}[section]
\newtheorem{proposition}[theorem]{Proposition}
\newtheorem*{theorem*}{Theorem}

\newtheorem{lemma}[theorem]{Lemma}

\theoremstyle{definition}

\newtheorem*{amalgamation*}{Amalgamation}

\newtheorem{remark}[theorem]{Remark}

\newtheorem*{remark*}{Remark}

\theoremstyle{plain}
\newcounter{algorithm}

\renewcommand{\int}{\operatorname{int}}



\begin{document}

	\title[Topological and bilipschitz types of surface singularities and links]{Topological and bilipschitz types of complex surface singularities and their links}

\date{\today}

\author[L. Fantini]{Lorenzo Fantini}
\address{Centre de Mathématiques Laurent Schwartz, Ecole Polytechnique, CNRS, Institut Polytechnique de Paris, Palaiseau, France}
\email{\href{mailto:lorenzo.fantini@polytechnique.edu}{lorenzo.fantini@polytechnique.edu}}
\urladdr{\url{https://lorenzofantini.eu/}}

\author[A. Pichon]{Anne Pichon}
\address{Aix Marseille Univ, CNRS, I2M, Marseille, France}
\email{\href{mailto:anne.pichon@univ-amu.fr}{anne.pichon@univ-amu.fr}}
\urladdr{\url{https://www.i2m.univ-amu.fr/perso/anne.pichon/}}

\begin{abstract}
In this paper, we prove that two normal complex surface germs that are inner bilipschitz---but not necessarily orientation-preserving---homeomorphic, have in fact the same oriented topological type and the same minimal plumbing graph.
Along the way, we show that the oriented homeomorphism type of an isolated complex surface singularity germ determines the oriented homeomorphism type of its link, providing a converse to the classical Conical Structure Theorem.
These results require to study the topology first, and the inner lipschitz geometry later, of Hirzebruch--Jung and cusp singularities, the normal surface singularities whose links are lens spaces and fiber bundles over the circle.
\end{abstract}

\maketitle


\hfill \textit{To David Trotman, in memoriam.}

\bigskip

\section{Introduction}

{\renewcommand*{\thetheorem}{\Alph{theorem}}  

Let $(X,0) \subset (\C^n,0)$ be a complex surface germ with an isolated singularity.
By the classical \emph{Conical Structure Theorem} (see for instance \cite[9.3.6]{BochnakCosteRoy1998}), given a radius $\epsilon>0$ small enough, the 4-manifold $X\cap B_\epsilon = \big\{x\in X\,\big|\,\norm{x}_{\C^n}\leq\epsilon\big\}$ is homeomorphic to the cone over the \emph{link} $\Link(X,0)=X\cap S_\epsilon = \big\{x\in X\,\big|\,\norm{x}_{\C^n}=\epsilon\big\}$ of $(X,0)$, which is a 3-manifold that is well-defined up to homeomorphism. 
Furthermore, $(X,0)$ has a natural orientation induced by its its complex structure, and its link  $\Link(X,0)$ is then also naturally oriented as the boundary of the $4$-manifold $X\cap B_\epsilon$.

As an immediate consequence of the Conical Structure Theorem, if the links of two isolated surface singularities $(X,0)$ and $(Y,0)$ are homeomorphic, then the germs are themselves homeomorphic, and 
the analogous result holds when 
keeping track of orientations: if two links are orientation-preserving homeomorphic then so are the original surface germs.
In this short paper we first prove that the converse also holds.
This consists of the implication $\ref{homeo} \Rightarrow \ref{link}$ appearing in the following more complete statement; see the discussion that follows it for a definition of the terms that we have not yet introduced.
Note that this implication holds in any dimension for hypersurface germs in $(\C^n,0)$, as was shown by Saeki \cite{Saeki1989}.
For simplicity, throughout the paper we will always implicitly assume that our singular germs are irreducible, that is, equivalently, that their links are connected.

\begin{theorem}
	\label{thm:oriented_topological type for complex surfaces} 
	Let $(X,0)$ and $(Y,0)$ be two  germs of complex surfaces with isolated singularities. 
	The following are equivalent: 
	\begin{enumerate}  
		\item \label{or_link} The links $\Link(X,0)$ and $\Link(Y,0)$ are orientation-preserving homeomorphic;
		\item \label{or_homeo} The germs $(X,0)$ and $(Y,0)$ are orientation-preserving homeomorphic;
		\item \label{or_sub-homeo} The germs $(X,0)$ and $(Y,0)$ are orientation-preserving subanalytically homeomorphic;
		\item \label{or_graph} The links $\Link(X,0)$ and $\Link(Y,0)$ have the same minimal plumbing graph.	\end{enumerate}
\end{theorem}

The implication $\ref{or_link} \Rightarrow \ref{or_homeo}$ comes from the Conical Structure Theorem, while the implication $\ref{or_link}\Rightarrow\ref{or_sub-homeo}$ is proven by Trotman \cite[Theorem~3.1]{Trotman2024}, as discussed in Remark~\ref{remark:Trotmans_proof} (the converse $\ref{or_sub-homeo} \Rightarrow \ref{or_link}$ being immediate). 
To discuss condition $\ref{or_graph}$, let us briefly recall the notion of \emph{plumbing graph} and the foundational results of Neumann.
In 1981, Neumann proved that the oriented link of an isolated complex surface singularity is an irreducible $3$-manifold (\cite[Theorem~1]{Neumann1981}).
Furthermore, it is a graph manifold in the sense of Waldhausen whose oriented homeomorphism class is determined by its so-called \emph{plumbing graphs}, which are the dual resolution graphs of the good resolutions of the singularity, decorated with the genera and self-intersections of the associated exceptional components.
Neumann then showed that, conversely, the data or the oriented homeomorphism class of the link determines its minimal plumbing structure up to isomorphism, and therefore its minimal plumbing graph, that is it determines the topology of the minimal good resolution of the singularity
(see \cite[Theorem~2]{Neumann1981}).
This is what we recorded as the equivalence $\ref{or_link} \Leftrightarrow \ref{or_graph}$ in the statement of Theorem~\ref{thm:oriented_topological type for complex surfaces}.
Note that in fact the oriented homeomorphism class of the link determines its minimal plumbing structure 
even up to isotopy, as proved by Popescu-Pampu \cite[Theorem~9.7]{Popescu-Pampu2007}.

Our first contribution in this paper is a proof of the implication $\ref{homeo} \Rightarrow \ref{link}$, which requires some recent results on the topology of 3-manifolds.
%
As far as we have been able to ascertain, this result is essentially known for the surface germs in $(\C^3,0)$, thanks to the already mentioned result of Saeki~\cite{Saeki1989}, and for those whose link is neither a lens space nor a torus bundle over a circle.
Indeed, outside of those two families of links, it can be obtained as a simple consequence of the results of Neumann \cite[Section~7]{Neumann1981}, combined with a theorem due to Waldhausen and Seifert--Threlfall classifying Seifert manifolds via their fundamental group (see Proposition~\ref{prop:special_case} for the complete argument).
To treat the general case, we combine the fact that a homeomorphism between the germs induces a $h$-cobordism between the two links (Lemma~\ref{lem:cobordism}) with a theorem of Turaev \cite[Theorem~1.4]{Turaev1988} stating that two $h$-cobordant geometric 3-manifolds are homeomorphic (see also \cite[Theorem~14]{Weber2018}).

As a fairly simple consequence of the first part of the proof of  $\ref{homeo} \Rightarrow \ref{link}$ of  Theorem~\ref{thm:oriented_topological type for complex surfaces}, we also have the following non-oriented version of  Theorem~\ref{thm:oriented_topological type for complex surfaces}.

\begin{theorem}  \label{thm:topological type for complex surfaces} 
	Let $(X,0)$ and $(Y,0)$ be two  germs of complex surfaces with isolated singularities. 
	The following are equivalent: 
	\begin{enumerate}
		\item \label{link} The links $\Link(X,0)$ and $\Link(Y,0)$ are homeomorphic;
		\item \label{homeo} The germs $(X,0)$ and $(Y,0)$ are homeomorphic;
		\item \label{sub-homeo} The germs $(X,0)$ and $(Y,0)$ are subanalytically homeomorphic.
	\end{enumerate}
\end{theorem}

However, as we will see in 
Proposition~\ref{prop:special_case}, the non oriented homeomorphism class of $(X,0)$ still determines the minimal plumbing graph of the link $\Link(X,0)$ except 
when $\Link(X,0)$ is a lens space or a torus bundle over the circle.
Now, a natural question is whether the (non oriented) homeomorphism class determines the minimal plumbing graph if we require the homeomorphism between the singular germs satisfies some additional requirement.
In Section~\ref{section:bi-lipschitz}, we prove that the answer to this question is positive, in the case of normal surfaces, if we impose that the homeomorphism is \emph{bilipschitz with respect to the inner metric}.
Recall that given an embedding of our germ $(X,0)$ into a smooth germ $(\C^N,0)$, the \emph{inner metric} of $(X,0)$ is the distance given by measuring the length of the shortest path in $X$ between two given points; it is well-defined (that is, independent of the embedding in a smooth germ) up to a \emph{bilipschitz homeomorphism}, which is a homeomorphism which is Lipschitz and whose inverse is itself Lipschitz.
We are now ready to state our final result.

\begin{theorem} 
	\label{thm:bilipschitz type for complex surfaces} 
	Let $(X,0)$ and $(Y,0)$ be two germs of normal complex surfaces and assume that $(X,0)$ and $(Y,0)$ are inner bilipschitz homeomorphic. 
	Then the germs $(X,0)$ and $(Y,0)$ are orientation-preserving homeomorphic, and therefore they have the same minimal plumbing graph.
\end{theorem}

Let us outline the proof of this theorem.
Due to the results of Neumann mentioned above, the only cases that we have to study are those of Hirzebruch--Jung singularities (whose links are lens spaces) and cusp singularities (whose links are torus bundles over the circle); the definitions of those are recalled at the beginning of Section~\ref{section:bi-lipschitz}.
Both those classes of singularities are \emph{taut} (in the sense of Laufer \cite{Laufer1973a}: two such singular germs are analytically isomorphic as soon as they have the same minimal plumbing graph).
In particular, their minimal resolutions factor through the blowup of their maximal ideal (Lemma~\ref{lem:taut_properties_1}), and they have reduced fundamental cycle. 

We deduce that the inner bilipschitz class of those singularities determines the minimal plumbing graph of their link, and therefore their orientation-preserving homeomorphism class by Theorem~\ref{thm:oriented_topological type for complex surfaces}.
This relies on the main result of Birbrair--Neumann--Pichon \cite[Theorem 1.9]{BirbrairNeumannPichon2014}, which establishes the inner bilipschitz invariance of a decomposition of a given normal complex surface germ in \emph{thick} and \emph{thin} zones, of some rational numbers called \emph{inner rates} that measure how quickly the volumes of the links of the thin parts shrink when approaching the singularity (see also \cite{BelottodaSilvaFantiniPichon2022a} for details), and of the foliations given by the Milnor fibers of a generic linear form.
More precisely, the thick-thin decomposition and the corresponding inner rates determine the dual graph of the exceptional divisors of the minimal good resolution of a Hirzebruch--Jung or cusp singularity, while the invariance of the homotopy type of the leaves of the foliation allows us to retrieve its self-intersection matrix.


We observe that orientation-reversing homeomorphisms may exist between singularities of the two classes described above. 
However, if such an orientation-reversing homeomorphism is bilipschitz, then Theorem~\ref{thm:bilipschitz type for complex surfaces} implies that another orientation-preserving homeomorphism must also exist (see Remark~\ref{remark:orientation_reversing_bilipschitz}).

Note that Theorem~\ref{thm:bilipschitz type for complex surfaces} requires the complex singularities to be normal and not only isolated: the difference between the two shows itself when considering metric properties, and is otherwise invisible to the topology.
Indeed, not only the normalization morphism of a isolated surface singularity is a homeomorphism, but it also induces a homeomorphism at the level of the links (this is because the normalization, being an analytic function, induces a rug function in the sense of Durfee \cite{Durfee1983a}, see also \cite{LeMaugendreWeber2001} or \cite[page 166]{LuengoPichon2005} for the details).
In fact, this is what allows us to apply the results of \cite{Neumann1981}, which are stated in the normal case, for isolated singularities in the context of Theorems~\ref{thm:oriented_topological type for complex surfaces} and \ref{thm:topological type for complex surfaces}.
On the other hand, the normalization is in general not bilipschitz, and we do not know if Theorem~\ref{thm:bilipschitz type for complex surfaces} still holds for non-normal isolated singularities.
We believe that a promising way to find a possible counterexample should come from the study of the inner rates associated with $\mathfrak M$-primary ideals initiated by Cherik \cite{Cherik2024}.

We do not expect the results we have obtained to extend to higher dimension;
we refer to Remark~\ref{remark:higher_dimension} for a discussion of the phenomenons that arise in that case.

}

\subsection*{Acknowledgments} 
We are very grateful to Michel Boileau, who showed us several references on the topology of lens spaces, to Patrick Popescu-Pampu, for his insightful comments on a previous version of this work, and to David Trotman, who kindly shared his preprint \cite{Trotman2024}.
We also thank Lev Birbrair for fruitful discussions and the anonymous referee for his thourough reading of our manuscript.
David Trotman kindly shared with us his preprint \cite{Trotman2024} a few months before his passing.
His work and ideas have been a great source of inspiration, and we deeply regret that we can no longer thank him personally.

Our work was partially supported by the project \emph{``Aspects tropicaux des singularités'' (SINTROP)} of the French \emph{Agence Nationale de la Recherche} (project ANR-22-CE40-0014).
A CC-BY public copyright license has been applied by the authors to the present document and will be applied to all subsequent versions up to the Author Accepted Manuscript arising from this submission, in accordance with the grant’s open access conditions and the authors' principles.

\section{Proof of Theorems~\ref{thm:oriented_topological type for complex surfaces} and~\ref{thm:topological type for complex surfaces}} 

As we mentioned in the introduction, what we really have to prove for both Theorems~\ref{thm:oriented_topological type for complex surfaces} and~\ref{thm:topological type for complex surfaces} is the implication $\ref{or_homeo} \Rightarrow \ref{or_link}$.
We begin, with the following result, by treating separately a special case for which that implication for Theorems~\ref{thm:oriented_topological type for complex surfaces} and \ref{thm:topological type for complex surfaces} (as well as the whole of Theorem~\ref{thm:bilipschitz type for complex surfaces}, in fact), can be proven simultaneously.
It is a simple consequence of results of Neumann, Waldhausen, and Seifert--Threlfall, and as such it is certainly known to the experts, but we believe it useful to record it as a separate statement and write down a complete proof.

\begin{proposition}\label{prop:special_case}
	Let $(X,0)$ and $(Y,0)$ be two locally homeomorphic germs of isolated complex surface singularities, and assume that the link $\Link(X,0)$ of $(X,0)$ is neither a lens space nor a torus bundle over the circle.
	Then the links $\Link(X,0)$ and $\Link(Y,0)$ are orientation-preserving homeomorphic.
\end{proposition}

\begin{proof}
As discussed at the end of the introduction, since the normalization gives a homeomorphism both at the level of germs and at the level of links, we can assume without loss of singularities that $(X,0)$ and $(Y,0)$ are normal.
By \cite[Theorem~3]{Neumann1981} the only normal surface singularity links such that by reversing the orientation we obtain another normal surface singularity link are lens spaces, which are links of Hirzebruch--Jung singularities, or torus bundles over the circle, which are links of cusp singularities.
Therefore, outside of these two special cases any homeomorphism between $\Link(X,0)$ and $\Link(Y,0)$ necessarily preserves the orientation.

Therefore, it remains to show that the two links are homeomorphic.
If $\phi \colon (X,0) \to (Y,0)$ is a local homeomorphism,
since $\phi(0) = 0$ then the map $\phi$ restricts to a local homeomorphism between $X \setminus \{0\}$ and $Y \setminus \{0\}$. 
As the germs $(X,0)$ and $(Y,0)$ are locally homeomorphic to the cones over their respective links, with vertex at $0$, and a punctured cone retracts by deformation onto its base, this implies that the links $L(X,0)$ and $L(Y,0)$ have isomorphic fundamental groups.

Now, as mentioned in the introduction, by \cite[Theorem~1]{Neumann1981} the singularity links $L(X,0)$ and $L(Y,0)$ are irreducible graph 3-manifolds in the sense of Waldhausen, with Seifert components having orientable bases.
In this category, except for lens spaces (which are the Seifert manifolds with finite cyclic fundamental group), the fundamental group determines the homeomorphism class (without keeping track of the orientations):
this follows from a result of Waldhausen (\cite[Corollary 6.5]{Waldhausen1968}) when the fundamental group is infinite, and from the computation of the fundamental group of such manifolds by Seifert--Threlfall (see the complete classification in the case of finite $\pi_1$ in \cite[\S6.2, Theorem~2]{Orlik1972}) for Seifert manifolds with finite but not cyclic fundamental group.
This completes the proof of the proposition.
\end{proof}

This proves that, outside of two special cases listed in the statement, part~$\ref{homeo}$ of Theorem~\ref{thm:topological type for complex surfaces} (which is weaker than~$\ref{or_homeo}$ of Theorem~\ref{thm:oriented_topological type for complex surfaces}) implies part~$\ref{or_link}$ of~\ref{thm:oriented_topological type for complex surfaces} (and therefore, also part~$\ref{link}$ of~\ref{thm:topological type for complex surfaces}).
To complete the proofs of the two theorems, we can now assume that $\Link(X,0)$ is either a lens space or a torus bundle over the circle.
Observe that, by the same arguments of classification via the fundamental group appearing in the proof of the lemma, the same is then true for $\Link(Y,0)$.
Furthermore, as in the proof of Proposition~\ref{prop:special_case}, we can assume that $(X,0)$ and $(Y,0)$ are normal.
This means that either $(X,0)$ and $(Y,0)$ are both Hirzebruch--Jung singularities, or that both are cusp singularities.
In either of the two cases, the oriented result, by which we mean the implication  $\ref{or_homeo} \Rightarrow \ref{or_link}$ of Theorem~\ref{thm:oriented_topological type for complex surfaces}, implies the non-oriented one.
Indeed, given any homeomorphism $\phi$ between the two isolated surface germs we can, if necessary, change the orientation of $(Y,0)$ (the fact that, if $(Y,0)$ belongs to one of the two classes of singularities we reduced ourselves to studying, then $(Y,0)$ with reversed orientation is still a singularity germ of the same class comes, again, from \cite[Theorem~3]{Neumann1981}) to assume that $\phi$ is orientation preserving, obtaining an (orientation-preserving, although that does not matter) homeomorphism between the two links $\Link(X,0)$ and $\Link(Y,0)$.

We reduced ourselves to proving the oriented result for lens spaces and torus bundles over the circle.
For this, we make use of the following topological result.
Let us first introduce some notation: if $M$ is an oriented topological manifold, we denote by $C(M)$ the topological cone over $M$, that is, $C(M) =M \times \R^+/ \{(x,0)\sim (y,0)\}$ and by $p_M \colon C(M) \to \R^+$ the map induced by the canonical projection on the second factor of $M \times  \R^+$.  
Given $r >0$, we set $M_r = p_M^{-1}(r )$ and  $B^{C(M)}_r = p_M^{-1}\big([0, r]\big)$.
Note that $C(M)$ admits a natural orientation which is induced by the one of $M$; every $M_r$ is then naturally oriented as the boundary of $B^{C(M)}_r\subset C(M)$.

\begin{lemma} \label{lem:cobordism}
	Let $M$ and $N$ be  two  $n$-dimensional oriented manifolds.
	Assume that there exists a orientation-preserving homeomorphism $\phi \colon C(M) \to C(N)$ such that $\phi(0)=0$.
	Let $\epsilon, \eta  \in \R^+$ such that $B^{C(N)}_{\epsilon}  \subset \int \big(\phi(B^{C(M)}_{\eta})\big)$.
	Then the $(n+1)$-dimensional manifold 
	\[
	\phi\big(B^{C(M)}_{\eta}\big)  \setminus \mathrm{int}\big(B^{C(N)}_{\epsilon}\big)
	\]
	realizes an oriented $h$-cobordism between its two boundary components $\phi(M_{\eta})$ and ${N}_{\epsilon}$. 
\end{lemma}

While it goes a bit farther than comparing the fundamental groups of the two links, this lemma is still quite elementary and is probably well-known.
However, we haven't found it in the literature and therefore we provide a complete proof.

\begin{proof}[Proof of Lemma~\ref{lem:cobordism}]  
	Set $W=  \phi(B^{C(M)}_{\eta})   \setminus \int  B^{C(N)}_{\epsilon} $.  
	Let us first prove that the inclusion $i \colon N_{\epsilon}   \hookrightarrow W$ is a homotopy equivalence. 
	For this,   let us choose  $\epsilon', \epsilon''$ and $\eta'$ such that $0<\eta' < \eta$,  $0<\epsilon'< \epsilon < \epsilon''$, and the following inclusions hold:
	\begin{equation*}
		B^{C(N)}_{\epsilon'}   \subset  \int\big(\phi(B^{C(M)}_{\eta'})\big),
		\quad			
		\phi(B^{C(M)}_{\eta'}) \subset  \int (B^{C(N)}_{\epsilon}),
		\quad
		\phi(B^{C(M)}_{\eta})  \subset   \int (B^{C(N)}_{\epsilon''}).		
	\end{equation*}
	See Figure~\ref{fig:cones} below for an illustration of the situation.
	%
	%
	%
	%
	%
	%
	%
	%
	%
	%
	%
	%
	%
	%
	%
	%
	%
	%
	%
	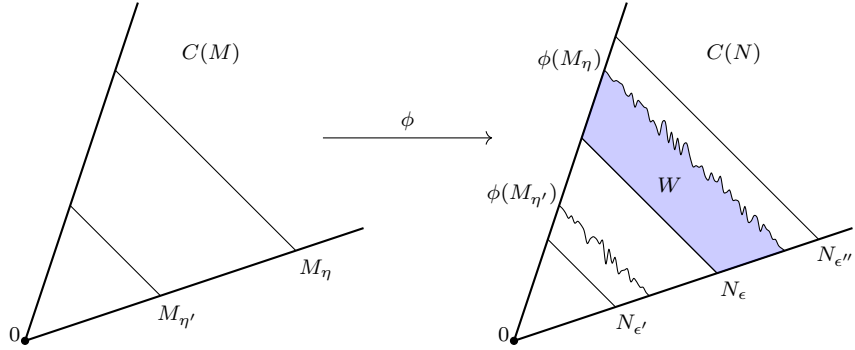
\begin{figure}[h]
		\begin{center}
			\begin{tikzpicture}
				\begin{scope}[scale=0.45]
					\draw[fill] (0,0)circle(3pt);
					
					\draw[thick](0,0)--+(10,10/3);					\draw[thick](0,0)--+(10/3,10);		
					
					\draw[thin] (4/3,4)--(4,4/3);
					\draw[thin] (8/3,8)--(8,8/3);
					
					\draw[thin,->] (8.8,6)--+(5,0);  
					
					\begin{footnotesize}
						\node(a)at(11.3,6.5){$\phi$};
						
						\node(a)at(-.3,.2){$0$};
						\node(a)at(5.5,8.5){$C(M)$};
						
						\node(a)at(4.5,.7){$M_{\eta'}$};
						\node(a)at(8.6,2.1){$M_{\eta}$};
					\end{footnotesize}

				\end{scope}

				\begin{scope}[xshift=6.5cm,scale=0.45]
					\draw[fill] (0,0)circle(3pt);
					
					\draw[thick](0,0)--+(10,10/3);					\draw[thick](0,0)--+(10/3,10);		
					
					\draw[thin] (1,3)--(3,1);
					\draw[thin] (2,6)--(6,2);
					\draw[thin] (3,9)--(9,3);

					\begin{scope}[xshift=1.3333cm]
						\draw plot [domain=0:8/3,samples=30,smooth] (\x,{4-\x+0.2*rand*\x*(8/3-\x)});
					\end{scope}
					
					\begin{scope}[xshift=2.6666cm]
						\draw[fill=blue,fill opacity=0.2] plot [domain=0:16/3,samples=60,smooth] (\x,{8-\x+0.05*rand*\x*(16/3-\x)})--(10/3,2)--(-2/3,6)--(0,8);
					\end{scope}
					
					\begin{footnotesize}
						\node(a)at(-.3,.2){$0$};
						\node(a)at(6.5,8.5){$C(N)$};
						
						\node(a)at(.25,4.3){$\phi(M_{\eta'})$};
						\node(a)at(1.65,8.3){$\phi(M_{\eta})$};
						
						\node(a)at(4.6,4.6){$W$};
						
						\node(a)at(3.5,.5){$N_{\epsilon'}$};
						\node(a)at(6.5,1.5){$N_{\epsilon}$};
						\node(a)at(9.5,2.5){$N_{\epsilon''}$};
					\end{footnotesize}
				\end{scope}
			\end{tikzpicture} 
		\end{center}
		\caption{An illustration of the construction used throughout the proof, with the subset $W$ in blue in the cone $C(N)$ on the right-hand side.}
		\label{fig:cones}
	\end{figure}
	Fix $k\geq 1$ and consider the morphism $(i)_* \colon \pi_k(N_{\epsilon} ) \to \pi_k(W)$ induced by $i$.
	Let $r \colon   B^{C(N)}_{\epsilon''} \setminus \int( B^{C(N)}_{\epsilon}) \longrightarrow  N_{\epsilon} $ be  the  radial flow retraction in the cone $C(N)$.  
	Since $r$  is a  strong deformation retraction, it  induces an  isomorphism  
	$$r_*\colon \pi_k\big(B^{C(N)}_{\epsilon''} \setminus \int( B^{C(N)}_{\epsilon})\big) \to  \pi_k(N_{\epsilon} ).$$
	Consider the two inclusions
	\begin{equation*} 
		N_{\epsilon} 
		\lhook\joinrel\xrightarrow{\;i\;\,} 
		W
		\lhook\joinrel\xrightarrow{\,i_2\,}  
		B^{C(N)}_{\epsilon''} \setminus \int(  B^{C(N)}_{\epsilon}).
	\end{equation*}
	Since $r$ is a left inverse for the composition $i_2 \circ i$, this implies that  the composition 
	\begin{equation*} 
		\pi_k(N_{\epsilon}  )
		\stackrel{(i)_*}{\longrightarrow} 
		\pi_k(W)
		\stackrel{(i_2)_*}{\longrightarrow} 
		\pi_k\big( B^{C(N)}_{\epsilon''} \setminus \int(  B^{C(N)}_{\epsilon})\big)
	\end{equation*}
	is the isomorphism $(r_*)^{-1}$. In particular, $(i)_*$ is injective. 
	
	Let us now prove that $(i)_*$ is surjective.  Consider the $4$-manifold
	\[
	W' = \phi(B^{C(M)}_{\eta'})   \setminus \int  B^{C(N)}_{\epsilon'}
	\]
	and the two inclusions
	\begin{equation*} 
		W' 
		\lhook\joinrel\xrightarrow{\,i_3\,} 
		B^{C(N)}_{\epsilon} \setminus \int ( B^{C(N)}_{\epsilon'})
		\lhook\joinrel\xrightarrow{\,i_4\,} 
		W \cup     B^{C(N)}_{\epsilon} \setminus \int ( B^{C(N)}_{\epsilon'}),
	\end{equation*}	
	and denote by $r'$ the radial flow retraction
	\[r' \colon  B^{C(M)}_{\eta} \setminus \int \big(\phi^{-1}(B^{C(N)}_{\epsilon'})\big) \longrightarrow  B^{C(M)}_{\eta'} \setminus \int \big(\phi^{-1}(B^{C(N)}_{\epsilon'})\big)
	\]
	in the cone $C(M)$. 
	Then the composition 
	\[
	R=\phi \circ r' \circ \phi^{-1} \colon W \cup    B^{C(N)}_{\epsilon} \setminus \int ( B^{C(N)}_{\epsilon'}) 
	\longrightarrow
	W'
	\]
	is a strong deformation retraction.
	Since $R$ is a left  inverse for the composition $i_4 \circ i_3$, this implies that for every $k \geq 0$ the induced morphism 
	\[
	(i_4)_* \colon \pi_k\big( B^{C(N)}_{\epsilon} \setminus \int ( B^{C(N)}_{\epsilon'})\big)
	\longrightarrow
	\pi_k( W \cup   B^{C(N)}_{\epsilon} \setminus \int ( B_{\epsilon'}^{C(N)})
	\] 
	is surjective. 
	
	Using again the radial flow in $C(N)$, we get two strong deformation retractions $r' \colon  B^{C(N)}_{\epsilon} \setminus \int ( B^{C(N)}_{\epsilon'})  \to  N_{\epsilon}$ and $r'' \colon  W \cup     B^{C(N)}_{\epsilon} \setminus \int ( B^{C(N)}_{\epsilon'}) \to W$. 
	Therefore, the two injections $i_5 \colon N_{\epsilon} {\hookrightarrow} B^{C(N)}_{\epsilon} \setminus \int ( B^{C(N)}_{\epsilon'})$ and $i_6 \colon W   {\hookrightarrow} W \cup B^{C(N)}_{\epsilon} \setminus \int ( B^{C(N)}_{\epsilon'})$ induce isomorphisms $(i_5)_*$ and $(i_6)_*$ at the level of homotopy groups $\pi_k$.
	Since $i_6 \circ i = i_4 \circ i_5$, then $(i)_* =  r''_* \circ ( i_4)_*  \circ (i_5)_*$ is surjective, and therefore it is an isomorphism. 
	As this holds for every $k\geq1$, by Whitehead's theorem  \cite[Theorem~4.5]{Hatcher2002} the map $i$ is a homotopy equivalence. 
	
	Similar arguments with respect to the cone $C(M)$ prove that the inclusion $j' \colon M_{\eta}  \hookrightarrow \phi^{-1}(W) $ induces an isomorphism at the level of homotopy groups $\pi_k$.  
	Therefore, the inclusion $j = \phi^{-1} \circ j' \circ \phi \colon \phi(M_{\eta} ) \hookrightarrow W$ is a homotopy equivalence, again by Whitehead's theorem.
	This shows that $W$ realizes an oriented $h$-cobordism between $\phi(M_{\eta})$ and $ N_{\epsilon} $.
\end{proof}

\begin{proof}[End of proof of the implication $\ref{or_homeo} \Rightarrow \ref{or_link}$ of Theorem~\ref{thm:oriented_topological type for complex surfaces}]
Since $L(X,0)$ and $L(Y,0)$ are irreducible by (\cite[Theorem 1]{Neumann1981}) and $h$-cobordant by Lemma~\ref{lem:cobordism}, then they are orientation-preserving homeomorphic by \cite[Theorem~1.4]{Turaev1988} (see also \cite[Theorem~14]{Weber2018}).
This concludes the proof of the implication $\ref{or_homeo} \Rightarrow \ref{or_link}$ of Theorem~\ref{thm:oriented_topological type for complex surfaces}, and therefore by what was argued before of Theorem~\ref{thm:topological type for complex surfaces} as well, completing the proof of the two theorems.
Note that in the case of lens spaces, we could have also applied a result of Atiyah and Bott \cite[page~479]{AtiyahBott1968}, or argued that two $h$-cobordant oriented lens spaces have the same Heegaard Floer $d$-invariant, and are therefore orientation-preserving homeomorphic by \cite{Greene2013} (more generally, Greene shows that two double-branched covers over alternating links are homeomorphic if and only if they have the same $d$-invariant); another proof can be found in \cite{DoigWehrli2015}.
\end{proof}

\begin{remark}
	\label{remark:higher_dimension}
	We do not know whether Theorems~\ref{thm:oriented_topological type for complex surfaces} and \ref{thm:topological type for complex surfaces} still hold for complex germs of higher dimension, besides in the hypersurface case treated by Saeki \cite{Saeki1989}.
	However, much care is certainly needed since in general there exist pairs of homeomorphic real analytic germs whose links are not homeomorphic, (see for instance \cite{King1976}), that is for which the condition \ref{homeo} does not imply \ref{link} nor \ref{sub-homeo}.
	This is related to the facts that two manifolds of real dimension greater than 3 may be or $h$-cobordant but not homeomorphic, or homeomorphic but nor diffeomorphic.
	In fact, concrete examples of the latter phenomenon can be constructed by taking suitable products of 3-dimensional lens spaces with a sphere $S^2$.
\end{remark}

\begin{remark}\label{remark:Trotmans_proof}
As we mentioned in the introduction, for both Theorems~\ref{thm:oriented_topological type for complex surfaces} and~\ref{thm:topological type for complex surfaces} the implication $\ref{or_link}\Rightarrow\ref{or_sub-homeo}$ is proven in \cite[Theorem~3.1]{Trotman2024}.
Since that reference is not easily available at this time, let us sketch Trotman's argument, which goes as follows: triangulating the links $\Link(X,0)$
and $\Link(Y,0)$ via semialgebraic homeomorphisms yields two homeomorphic piecewise-linear 3-manifolds, which are piecewise-linear isomorphic by Moise’s theorem \cite{Moise1952}. 
We deduce a semialgebraic homeomorphism between $\Link(X,0)$
and $\Link(Y,0)$, and therefore a semi-algebraic
homeomorphism between the cones over the two links.
This gives a semialgebraic (and therefore subanalytic) homeomorphism between the germs $(X,0)$ and $(Y,0)$.
\end{remark}

\section{Bilipschitz homeomorphisms and orientation}
\label{section:bi-lipschitz}

Let us now move to the proof of Theorem~\ref{thm:bilipschitz type for complex surfaces}.
We can restrict our attention to the normal complex surface singularities that admit orientation-reversing homeomorphisms.
By the result of Neumann~\cite[Theorem~3]{Neumann1981} already used in the proof of Proposition~\ref{prop:special_case}, this means that we can assume that $(X,0)$ and $(Y,0)$ are both Hirzebruch--Jung singularities, or that both are cusp singularities.
Let us recall the precise definitions of these classes of singularities and all the facts that we need.

A \emph{Hirzebruch--Jung singularity} is a normal complex surface germ that can be obtained as the normalization of the hypersurface singularity defined by the equation $z^p-x^{p-q}y=0$ in $\C^3$, for some coprime integers $p>q\geq 1$.
We refer to \cite{Hirzebruch1953,HirzebruchNeumannKoh1971,Popescu-Pampu2007,Weber2018} for abundant details, and limit ourselves to describing what we need in the current paper.
Hirzebruch--Jung singularities are minimal singularities, so that, in particular, they are rational.
If $(X,0)$ is the Hirzebruch--Jung singularity obtained from the equation $z^p-x^{p-q}y=0$ above, then the link $\Link(X,0)$ of $(X,0)$ is the 3-dimensional lens space $L(p,q)$ (see for instance \cite[Theorem~22]{Weber2018}).
Consider the integers $b_1,b_2,\ldots,b_k\geq 2$ such that
\begin{equation}\label{eqn:continued_fraction_lens_spaces}
\frac{p}{q} = b_1-\cfrac{1}{b_2 -\cfrac{1}{\ddots - \cfrac{1}{b_k}}}\ .
\end{equation}
The minimal plumbing graph of $L(p,q)$ consists then of a chain of vertices of self-intersections $-b_1,-b_2,\ldots,-b_k$, all corresponding to rational curves, as is pictured in Figure~\ref{figure:example_lens_graph} below.
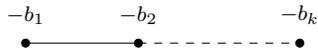
\begin{figure}[h]
	\begin{center}
		\begin{tikzpicture}
			
			\begin{scope}[scale=1.5]
	
	\draw[thin] (0,0)--(1,0);
	\draw[thin,dashed] (1,0)--(2.43,0);
	
	\draw[fill] (0,0)circle(1pt);
	\draw[fill] (1,0)circle(1pt);
	\draw[fill] (2.43,0)circle(1pt);
	
	\begin{footnotesize}
		\node(a)at(0,.25){$-b_1$};
		\node(a)at(1,.25){$-b_2$};
		\node(a)at(2.43,.25){$-b_k$};
		
		
	\end{footnotesize}
	
\end{scope}
%
%
%
%
%
%
%
%
%
%
		\end{tikzpicture} 
	\end{center}
	\caption{Minimal plumbing graph of the lens space $L(p,q)$.}
	\label{figure:example_lens_graph}
\end{figure}
Hirzebruch--Jung singularities are \emph{taut} (see \cite[Section~2.2, type~II]{Laufer1973a}, but the result goes back to Brieskorn \cite{Brieskorn1967/1968}), which means that the minimal plumbing graph of their link determines the analytic type of the singularity.
A lens space $L(p,q)$ is always homeomorphic to the lens space $L(p,p-q)$ via an orientation-reversing homeomorphism.
We note that one can describe completely the minimal plumbing graph of $L(p,p-q)$ in terms of the one of $L(p,q)$ by comparing the expansions of $p/q$ and $p/(p-q)$ as fractions as in equation~\eqref{eqn:continued_fraction_lens_spaces}; this is done in detail in \cite[Section~7]{Neumann1981}.

Similarly, \emph{cusp singularities} are the normal complex surface germs that are taut and whose minimal plumbing graph is a cycle of vertices of genus 0 with negative-definite self-intersection matrix (that is, equivalently, all vertices have self-intersection $\leq -2$ and at least one has self-intersection $\leq -3$).
Their links are precisely the surface singularity links that are torus bundles over the circle (observe that the monodromy of such bundles must then have trace at least 3---an hypothesis that is explicitly spelled out throughout \cite{Neumann1981}---as the other torus bundles over the circle cannot be realized as links of surface singularities).
We refer to \cite{Hirzebruch1973,Laufer1977,Popescu-Pampu2007} for details about cusp singularities, and to \cite[Section~2.2, type~V]{Laufer1973a} for their tautness.

\

Let us briefly recall the definitions of fundamental cycle and maximal cycle of a normal complex surface germ $(X,0)$; we refer to \cite[Section~2]{BelottodaSilvaFantiniNemethiPichon2022} for more detailed definitions 
and further references.
Let $\pi\colon(X_\pi,E_\pi)\to(X,0)$ be a good resolution of 
$(X,0)$, and denote by $\Gamma_\pi$ the dual graph of the exceptional divisor $E_\pi=\bigcup_{i\in I}E_{i}=\pi^{-1}(0)$ of $\pi$, endowed as usual with the genera $g(E_i)$ and self-intersections $E_i^2$ of the irreducible components of $E_\pi$.
Then by a classical result of Artin there exists a unique smallest divisor $Z_{\min}^{\Gamma_\pi}=\bigcup_{i\in I}a_iE_{i}$ supported on $E_\pi$ and such that $Z_{\min}^{\Gamma_\pi}\cdot E_{i}\leq0$ for every ${i\in I}$.
The divisor $Z_{\min}^{\Gamma_\pi}$ is called the \emph{fundamental cycle} of $(X,0)$ (or of $\Gamma_\pi$, since it only depends on the weighted graph), and it is positive, that is we have $a_i>0$ for every $i$.
Another relevant divisor supported on $E_\pi$ is the \emph{maximal cycle} of $(X,0)$, which is defined as $Z_{\max}^{\Gamma_\pi}(X,0)=\bigcup_{i\in I}m_iE_i$, where $m_i$ is the multiplicity of a generic linear form $h\colon (X,0)\to(\C,0)$ along the irreducible component $E_i$ of $E_\pi$ (that is, $m_i$ is the multiplicity of the maximal ideal of $(X,0)$ along $E_i$; said otherwise, $Z_{\max}^{\Gamma_\pi}(X,0)$ is the scheme-theoretic fiber $\pi^{-1}(0)$ of $\pi$ at $0$).
In general, $Z_{\max}^{\Gamma_\pi}(X,0)$ may be strictly bigger than the fundamental cycle $Z_{\min}^{\Gamma_\pi}$; let us show that this is not the case for taut singularities.

\begin{lemma}\label{lem:taut_properties_1}
	Let $(X,0)$ be a taut normal surface singularity and let $\pi$ be the minimal good resolution of $(X,0)$.
	Then:
	\begin{enumerate}
		\item The fundamental cycle $Z_{\min}^{\Gamma_\pi}$ and the maximal cycle $Z_{\max}^{\Gamma_\pi}(X,0)$ coincide.
		\item The resolution $\pi$ factors through the blowup of the maximal ideal of $(X,0)$;
	\end{enumerate}
\end{lemma}

\begin{proof}
Set $Z_{\min}^{\Gamma_\pi} = \sum_{i\in I} a_i E_i$. 
For each $i \in I$, set $k_i = -E_i^2 a_i + \sum_{j \neq i} a_j E_i \cdot E_j$. 
Then, by Laufer's algorithm, we have $k_i \geq 0$ for all $i$, and, since the intersection matrix $(E_i\cdot E_j)$ is negative definite, there exists $i \in I$ with $k_i >0$. 
Now, consider the graph $ \Gamma_{\pi}$ as the plumbing graph of an  oriented plumbed  $3$-manifold $M$ homeomorphic to the link $L(X,0)$, and, for each $i$, attach $k_i$ arrows representing $k_i$ disjoint $S^1$-fibers of the $S^1$-bundle over $E_i$ defining the plumbing to the corresponding vertex of $\Gamma_\pi$.
Let  $K \subset M$ be the link  whose components are these $\sum_{i \in I} k_i$ circles.
Then $L$ is a fibered link in $M$ with a surface fiber intersecting positively each $S^1$-fiber of the plumbing, and by \cite[Theorem~5.4]{Pichon2001} (see also \cite[Theorem 2.1, (ii)$\Rightarrow$(i)]{NeumannPichon2007}), this implies that  there exists a normal surface germ $(Y,0)$ and a holomorphic function germ $f \colon (Y,0) \to (\C,0)$ such that the pair  $(M,L)$ is homeomorphic to the pair of links $\big(L(Y,0), f^{-1}(0)\big)$. Since $(X,0)$ is taut,  any such  $(Y,0)$  is analytically equivalent to $(X,0)$, and then, we can take $(Y,0) = (X,0)$ and $f \colon (X,0) \to (\C,0)$. 
Set $Z(f)= \sum_{i \in I} m_i(f) E_i$, where $m_i(f)$ is the multiplicity of $f$ along $E_i$. 
By definition of $f$, we have $m_i(f)=a_i$ for all $i$, that is, $Z(f) = Z_{\min}^{\Gamma_\pi} $. 
Since $Z_{\max}^{\Gamma_\pi}(X,0) \leq Z(f)$ by definition of the maximal cycle, and $Z_{\min}^{\Gamma_\pi} \leq Z_{\max}^{\Gamma_\pi}(X,0)$ by definition of the fundamental cycle, we deduce that 
$Z_{\max}^{\Gamma_\pi}(X,0) = Z_{\min}^{\Gamma_\pi}$. 

Assume now that  $\pi$ does not factor through the blowup of the maximal divisor of $(X,0)$.
Then there exists a base point $p \in E_\pi$ for the family of generic linear forms on $(X,0)$. 
Let $h \colon (X,0) \to (\C,0)$ be a generic linear form on $(X,0)$, so that we have $Z(h) = Z_{\max}^{\Gamma_\pi}(X,0) = Z_{\min}^{\Gamma_\pi} = \sum a_i E_i $, and let $L \subset M$ be its link. 
Let us consider the link $L \cup L'$ in $M$ obtained by  taking $2k_i$ components  instead of $k_i$ for each $i$. 
Again by \cite[Theorem~5.4]{Pichon2001}, there exists a function $g \colon (X,0) \to (\C,0)$ whose link is $(M,L \cup L')$. 
Since for every $i$ we have $m_i(g) - m_i(h) =2a_i-a_i=a_i>0$, then the quotient $f=g/h$ defined outside $0$ extends to a holomorphic germ $f \colon (X,0) \to (\C,0)$ whose strict transform by $\pi$ does not intersect the strict transform of $h$ and whose link is homeomorphic to $(M,L')$. 
In particular, $Z(f) = Z(h)$, but the strict transform of $f$ by $\pi$ does not pass through $p$. 
Therefore we have $m_{E'}(f) <  m_{E'}(h)$, where $E'$ is the exceptional curve  of the  blowup of the base point $p$. 
This contradicts the fact that $h$ is a generic linear form on $(X,0)$. 
\end{proof}

Given the specific shape of the minimal resolution graphs of the surface singularities we are working with, it is easy to verify that the fundamental cycle $Z_{\min}^{\Gamma\pi}$ of $\Gamma_\pi$ is reduced (this can be seen as the first step of the classical Laufer's algorithm to find $Z_{\min}^{\Gamma_\pi}$).
Therefore, we immediately deduce from Lemma~\ref{lem:taut_properties_1} the following fact, which we record as a separate statement for convenience. 

\begin{lemma}\label{lem:taut_properties_2}
	Let $(X,0)$ be a Hirzebruch--Jung or cusp surface singularity and let $\pi$ be the minimal good resolution of $(X,0)$.
	Then the maximal cycle $Z_{\max}^{\Gamma_\pi}(X,0)$ of $(X,0)$ is reduced.
\end{lemma}

To prove Theorem~\ref{thm:bilipschitz type for complex surfaces}, we need to show that the inner bilipschitz class of $(X,0)$ determines the dual resolution graph $\Gamma_\pi$ associated with the minimal good resolution $\pi\colon (X_\pi,E_\pi)\to (X,0)$ of $(X,0)$, weighted with its self-intersections (the genera of all exceptional components being equal to zero for Hirzebruch--Jung and cusp singularities).
Let us describe in detail what the complete inner bilipschitz invariant of Birbrair--Neumann--Pichon \cite[Theorem 1.9]{BirbrairNeumannPichon2014} consists of in the case of our singularities.

Recall that the \emph{$\mathcal L$-nodes} of $(X,0)$ are those vertices of $\Gamma_\pi$ that correspond to the exceptional components of the blowup of the maximal ideal of $(X,0)$.
Those are precisely the vertices corresponding to the components $E_i$ of $E_\pi$ that satisfy $E_i \cdot Z_{\max}^{\Gamma_\pi}(X,0)<0$.
Since the maximal cycle is reduced by Lemma~\ref{lem:taut_properties_2}, those are precisely the vertices with self-intersection less than or equal to $-3$, as well as, in the Hirzebruch--Jung case, the two vertices of valency 1 (the \emph{ends} of the graph).
Denote by $\nu_0, \ldots, \nu_s$ the $\mathcal L$-nodes of $(X,0)$, with the convention that they are ordered linearly along the graph, with $\nu_0$ and $\nu_s$ being the two ends in the Hirzebruch--Jung case, and $\nu_0=\nu_s$ in the cusp case.
Consider the integers $n_1,\ldots,n_s\geq0$ such that in the graph $\Gamma_\pi$ two consecutive $\mathcal L$-nodes $\nu_{i-1}$ and $\nu_i$  are connected by a string consisting of $n_i+1$ edges, containing in its interior $n_i\geq0$ vertices of self-intersection $-2$.
Denote by $\tilde\pi \colon (X_{\tilde\pi},E_{\tilde\pi})\to (X,0)$ the good resolution of $(X,0)$ obtained from $\pi$ as follows: for every $i=1,\ldots,s$, if $n_i$ is even, blow up once the point of $\pi^{-1}(0)$ associated with the edge at the middle of the string connecting $\nu_{i-1}$ to $\nu_i$.
This sequence of blowups modifies the resolution graph $\Gamma_\pi$, yielding a graph $\Gamma_{\tilde\pi}$ such that every string between adjacent $\mathcal L$-nodes $\nu_{i-1}$ and $\nu_i$ has a central vertex $w_i$.
The vertices $w_1,\ldots,w_s$ of $\Gamma_{\tilde\pi}$ are the so-called \emph{special $\mathcal P$-nodes} of $(X,0)$.
With each special $\mathcal P$-node $w_i$ is associated its \emph{inner rate}, which is the rational number $(n_i+3)/2$ (see Examples 15.5 and 15.6 in \cite{BirbrairNeumannPichon2014}).

Let us now describe the so-called {thick-thin decomposition} of $(X,0)$.
For each component $E_\nu$ of $E_{\tilde\pi}$ let $N(E_\nu)$ be a small closed tubular neighborhood of $E_{\nu}$ in $X_{\tilde\pi}$.  
The \emph{thick parts} of $(X,0)$ are then the $s$ subgerms $Y_i=\tilde\pi\big(N(E_{\nu_i})\big)$, for $i=1,\ldots,s$, where we denoted by $E_{\nu_i}$ the exceptional component of $E_{\tilde\pi}$ associated with the $\mathcal L$-node $\nu_i$.
The \emph{thin parts} of $(X,0)$ are the subgerms $Z_1,\ldots,Z_s$ obtained, in a similar way, as follows:
if we denote by $E_1,\ldots,E_{l_i}$ the components of $E_{\tilde\pi}$ corresponding to the $(-2)$-curves in the interior of the string connecting $\nu_{i-1}$ to $\nu_i$, then
\(
Z_i = \tilde\pi\Big(\overline{\big(\bigcup_{j=1}^{l_i}N(E_j)\big)\setminus\big(N(E_{\nu_{i-1}})\cup N(E_{\nu_{i}})\big)}\Big).
\)
The decomposition of the germ $(X,0)$ as the union of the semi-algebraic germs 
\(
(X,0) = \bigcup_{i=0}^s (Y_i,0) \cup \bigcup_{i=1}^s(Z_i,0),
\)
is called in \cite{BirbrairNeumannPichon2014} the \emph{thick-thin decomposition} of $(X,0)$.

Given a small $\epsilon>0$, the thick-thin decomposition induces a decomposition of the link $X^{(\epsilon)}=X\cap S_\epsilon \cong \Link(X,0)$ as the union $X^{(\epsilon)} = \bigcup_{i=0}^s Y_i^{(\epsilon)} \cup  \bigcup_{i=1}^s Z_i^{(\epsilon)}$.
Now, let $h \colon (X,0) \to (\C,0)$ be a generic linear form, let $K = X^{\epsilon} \cap h^{-1}(0)$ be its link and let $\phi = \frac{h}{|h|} \colon  X^{(\epsilon)} \setminus K \to S^1$ be its Milnor fibration.
It is a locally trivial fibration which defines an open book decomposition of $X^{(\epsilon)} $ with binding the link $K$.
By genericity of $h$, the curve germ $\big(h^{-1}(0),0\big)$ is contained in the thick part of $(X,0)$, so that $K$ is contained in the union $\bigcup_{i=0}^s Y_i^{(\epsilon)}$ and, for each $i=1, \ldots, s$, the restriction $\zeta_i  = \phi_{| Z_i^{(\epsilon)} }\colon  Z_i^{(\epsilon)} \to S^1$ of $\phi$ to $Z_i^{(\epsilon)}$ is a locally trivial fibration.

The Classification Theorem of Birbrair--Neumann--Pichon \cite[Theorem 1.9]{BirbrairNeumannPichon2014}, specialized to our singularities, tells us that the inner bilipschitz class of $(X,0)$ determines (and is determined by) the  following data:
\begin{enumerate}
	\item The decomposition, up to homeomorphism, of the  link $X^{(\epsilon)}$ as the union $X^{(\epsilon)} = \bigcup_{i=0}^s Y_i^{(\epsilon)} \cup  \bigcup_{i=1}^s Z_i^{(\epsilon)}$; 
	\item\label{condition_foliation} For each $i = 1, \ldots, s$, the homotopy class of the foliation by fibers of the map $\zeta_i  \colon  Z_i^{(\epsilon)}  \to S^1$; 
	\item For each $i=1,\ldots, s$, the inner rate of the $P$-node $w_i$.
\end{enumerate}

Since the inner rate of the $P$-node $w_i$ is equal to $\frac{n_i + 3}{2}$, the inner bilipschitz class of $(X,0)$ determines each $n_i$.

It remains to show that the bilipschitz class of $(X,0)$ also determines the self-intersection of the exceptional curves $E_{\nu_1},\ldots,E_{\nu_s}$ in $X_{\tilde\pi}$ (and therefore in $X_{\pi}$, since we know explicitly how to pass from $\pi$ to $\tilde\pi$) for each $L$-node $\nu_i$.
This will be given to us by the datum \ref{condition_foliation} of the Classification Theorem above.

\

Denote by $e_i\leq-2$ the self-intersection of the exceptional curve $E_{\nu_i}$ in $X_{\tilde\pi}$.
Since the maximal cycle of $Z_{\max}^{\Gamma_\pi}(X,0)$ is reduced by Lemma~\ref{lem:taut_properties_2}, then the curve $E_{\nu_i}$ is reduced in $X_{\tilde{\pi}}$ as well, and the part of the strict transform of $h=0$ to $X_{\tilde{\pi}}$ that intersects $E_{\nu_i}$ consists of exactly $-e_i-v_i$ curvettes of $E_{\nu_i}$, where $v_i \in \{1,2\}$ is the valency of the vertex $\nu_i$ in $\Gamma_{\tilde{\pi}}$.
Therefore, $K_i = K \cap Y_i^{(\epsilon)}$ consists of $-e_i-v_i$ isotopic oriented simple closed curves in the thickened torus $Y_i^{(\epsilon)}$.

Let us first treat the case where $v_i=2$, that is where $\nu_i$ is not an end of the graph $\Gamma_{\tilde{\pi}}$.
Then $Y_i^{(\epsilon)}$ is a thickened torus $S^1 \times S^1 \times I$ which has one boundary component in common with $Z_{i-1}^{(\epsilon)}$ and the other one in common with $Z_{i}^{(\epsilon)}$.
Fix  $t \in S^1$.
Then $F_i= \phi^{-1}(t) \cap Y_i^{(\epsilon)}$  is a genus zero oriented surface with  boundary  the union  $\partial F_i = K_i \cup \gamma_{i}^- \cup \gamma_i^+$, where we write $\gamma_{i}^-=F_i  \cap Z_{i-1}^{(\epsilon)}$ and $\gamma_{i}^+=F_i  \cap Z_{i}^{(\epsilon)}$.
Therefore, the class of $\gamma_{i}^- + \gamma_{i}^+ + K_i$ vanishes in $\pi_1(Y_i^{(\epsilon)})$.

Now, since $Z_{\max}^{\Gamma_\pi}(X,0)$ is reduced, then $\phi^{-1}(t) \cap Z_{i-1}^{(\epsilon)}$ and  $\phi^{-1}(t) \cap Z_{i}^{(\epsilon)}$ are connected (they are in fact annuli) and therefore, by the datum~\ref{condition_foliation} of the Classification Theorem, the bilipschitz geometry of $(X,0)$ determines the homotopy class of their (connected) boundary curves $\gamma_{i}^-$ and $\gamma_{i}^+$, and therefore their homotopy classes in the intermediate thickened torus  $Y_i^{(\epsilon)}$. 
Since the class of $\gamma_{i}^- + \gamma_{i}^+ + K_i$ vanishes in $\pi_1(Y_i^{(\epsilon)})$, this implies that the bilipschitz geometry of $(X,0)$ determines the homotopy class of $K_i$ in  $Y_i^{(\epsilon)}$, and therefore the number $- e_i- 2$ of its connected components.
Therefore, the bilipschitz geometry of $(X,0)$ determines the self-intersection $e_i$. 

The case where $v_i=1$, that is where $(X,0)$ is Hirzebruch--Jung and $\nu_i$ is an end of its resolution graph, uses similar arguments.
Say, without loss of generality, that $i=0$.
Then $K_0 = K \cap Y_0^{(\epsilon)}$ consists of $-e_0-1$ isotopic oriented simple closed curves in the solid torus $Y_0^{(\epsilon)} \cong S^1 \times D^2$ and  $ Y_0^{(\epsilon)}$ intersects $Z_{1}^{(\epsilon)}$ along their common torus boundary component.
Then $F_0= \phi^{-1}(t) \cap Y_0^{(\epsilon)}$ has as boundary the union $K_0 \cup \gamma_0^+$, where $\gamma_0^+ $ is the connected curve $F_0 \cap Z_{1}^{(\epsilon)}$ whose homotopy class is determined by the bilipschitz geometry of $(X,0)$.
Then the bilipschitz geometry of $(X,0)$ also determines the homotopy class of $K_0$ in $Y_0^{(\epsilon)}$, and therefore the number $- e_0- 1$ of its connected components as well as the self-intersection number $e_0$.
This completes the proof of Theorem~\ref{thm:bilipschitz type for complex surfaces}.
\hfill\qed
%

\begin{remark}
	In order to prove Theorem~\ref{thm:bilipschitz type for complex surfaces} it is really necessary to use the invariance of the homotopy type of the foliations given by the Milnor fibers of a generic linear form.
	Indeed, let $(X,0)$ and $(Y,0)$ be the Hirzebruch--Jung singularities obtaining by normalizing the subgerms of $(\C^3,0)$ defined by the equations $z^{27}-x^{19}y=0$ and $z^{27}-x^{8}y=0$ respectively.
	Then the link $\Link(X,0)$ of $(X,0)$ is the lens space $L(27,8)$, while the link $\Link(Y,0)$ of $(Y,0)$, which is the lens space $L(27,19)=L(27,27-8)$, is homeomorphic to $\Link(X,0)$ via an orientation-reversing homeomorphism. 
	The minimal plumbing graphs of $\Link(X,0)$ and $\Link(X,0)$ are pictured in Figure~\ref{figure:example_explicit_lens_graphs} below.
	\begin{figure}[h!]
		\begin{center}
			\begin{tikzpicture}			
				\begin{scope}[scale=1.5]
					
					\draw[thin,>-stealth,->](0,0)--+(-.2,0.4);		
					\draw[thin,>-stealth,->](0,0)--+(0,0.4);		
					\draw[thin,>-stealth,->](0,0)--+(.2,0.4);		
					
					\draw[thin,>-stealth,->](2,0)--+(0,0.4);
					\draw[thin,>-stealth,->](3,0)--+(0,0.4);
					
					\draw[thin] (0,0)--(3,0);
					
					\draw[fill] (0,0)circle(1pt);
					\draw[fill] (1,0)circle(1pt);
					\draw[fill] (2,0)circle(1pt);
					\draw[fill] (3,0)circle(1pt);

					\begin{footnotesize}
						\node(a)at(0,-.25){$-4$};
						\node(a)at(1,-.25){$-2$};
						\node(a)at(2,-.25){$-3$};
						\node(a)at(3,-.25){$-2$};
					\end{footnotesize}

				\end{scope}
				
				\begin{scope}[xshift=7.2cm,scale=1.5]
					
					\draw[thin,>-stealth,->](2,0)--+(-.1,0.4);		
					\draw[thin,>-stealth,->](2,0)--+(.1,0.4);		
					
					\draw[thin,>-stealth,->](0,0)--+(0,0.4);
					
					\draw[thin,>-stealth,->](3,0)--+(-.1,0.4);
					\draw[thin,>-stealth,->](3,0)--+(.1,0.4);
					
					\draw[thin] (0,0)--(3,0);
					
					\draw[fill] (0,0)circle(1pt);
					\draw[fill] (1,0)circle(1pt);
					\draw[fill] (2,0)circle(1pt);
					\draw[fill] (3,0)circle(1pt);									
					
					\begin{footnotesize}
						\node(a)at(0,-.25){$-2$};
						\node(a)at(1,-.25){$-2$};
						\node(a)at(2,-.25){$-4$};
						\node(a)at(3,-.25){$-3$};
					\end{footnotesize}

				\end{scope}
			\end{tikzpicture} 
		\end{center}
		\caption{Minimal plumbing graphs of the Hirzebruch--Jung singularities $(X,0)$, on the left, and $(Y,0)$, on the right.}
		\label{figure:example_explicit_lens_graphs}
	\end{figure}
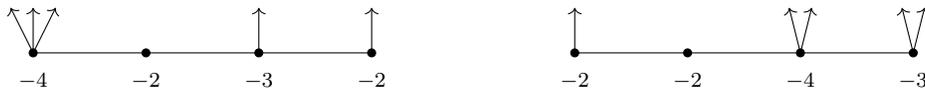
	We observe that both $(X,0)$ and $(Y,0)$ contain three thick zones (corresponding to the three exceptional components of the blowup of their maximal ideals) interpolated by two thin zones.
	Furthermore, it is easy to check that for both singularities the maximal inner rates of the two thin zones are $3/2$ and $2$ (for instance, one can apply \cite[Proposition~5.1]{BelottodaSilvaFantiniPichon2022}).
	An analogous example with two homeomorphic cusp singularities can be obtained from a careful study of the explicit description of \cite[Section~7]{Neumann1981}, see for instance Figure~\ref{figure:example_explicit_cusps_graphs} below.
	\begin{figure}[h!]
		\begin{center}
			\begin{tikzpicture}			
				\begin{scope}[scale=1]
					\draw[thin,>-stealth,->](5,0)--+(0,0.6);		
					
					\draw[thin,>-stealth,->](2,0)--+(-.3,0.6);		
					\draw[thin,>-stealth,->](2,0)--+(0,0.6);		
					\draw[thin,>-stealth,->](2,0)--+(.3,0.6);	
					
					\draw[thin,>-stealth,->](0,0)--+(-.15,0.6);
					\draw[thin,>-stealth,->](0,0)--+(.15,0.6);	
					
					\draw[thin] (0,0)--(5,0);
					\draw[thin] (0,0) to[out=-50,in=230] node [sloped,below] {} (5,0);
					
					\draw[fill] (0,0)circle(1.5pt);
					\draw[fill] (1,0)circle(1.5pt);
					\draw[fill] (2,0)circle(1.5pt);
					\draw[fill] (3,0)circle(1.5pt);
					\draw[fill] (4,0)circle(1.5pt);
					\draw[fill] (5,0)circle(1.5pt);

					\begin{footnotesize}
						\node(a)at(-.2,-.25){$-4$};
						\node(a)at(1,-.25){$-2$};
						\node(a)at(2,-.25){$-5$};
						\node(a)at(3,-.25){$-2$};
						\node(a)at(4,-.25){$-2$};
						\node(a)at(5.1,-.25){$-3$};
					\end{footnotesize}

				\end{scope}
				
				\begin{scope}[xshift=6.6cm,scale=1]
					\draw[thin,>-stealth,->](0,0)--+(0,0.6);		
					
					\draw[thin,>-stealth,->](5,0)--+(-.3,0.6);		
					\draw[thin,>-stealth,->](5,0)--+(0,0.6);		
					\draw[thin,>-stealth,->](5,0)--+(.3,0.6);	
					
					\draw[thin,>-stealth,->](2,0)--+(-.15,0.6);
					\draw[thin,>-stealth,->](2,0)--+(.15,0.6);	
					
					\draw[thin] (0,0)--(5,0);
					\draw[thin] (0,0) to[out=-50,in=230] node [sloped,below] {} (5,0);
					
					\draw[fill] (0,0)circle(1.5pt);
					\draw[fill] (1,0)circle(1.5pt);
					\draw[fill] (2,0)circle(1.5pt);
					\draw[fill] (3,0)circle(1.5pt);
					\draw[fill] (4,0)circle(1.5pt);
					\draw[fill] (5,0)circle(1.5pt);

					\begin{footnotesize}
						\node(a)at(-.2,-.25){$-3$};
						\node(a)at(1,-.25){$-2$};
						\node(a)at(2,-.25){$-4$};
						\node(a)at(3,-.25){$-2$};
						\node(a)at(4,-.25){$-2$};
						\node(a)at(5.1,-.25){$-5$};
					\end{footnotesize}

				\end{scope}
			\end{tikzpicture} 
		\end{center}
		\caption{The minimal plumbing graphs of two cusp singularities that are orientation reversing-homeomorphic, while having the same number of thick and thin zones and the same maximal inner rates.
		Note that these singularities are not hypersurfaces in $\C^3$.}
		\label{figure:example_explicit_cusps_graphs}
	\end{figure}
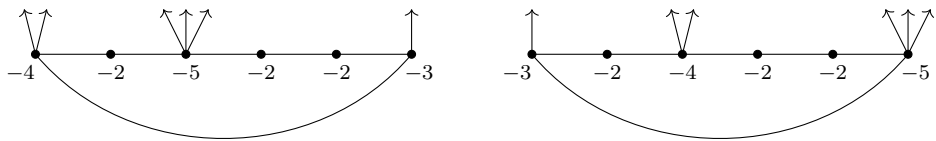
\end{remark}

\begin{remark}
	Hirzebruch--Jung singularities are minimal singularities, and therefore they are Lipschitz Normally Embedded by \cite[Theorem~1.2]{NeumannPedersenPichon2020b} (see \cite[Section~2.2]{FantiniPichon2023} for further comments).
	Some of the arguments of this section could then be replaced by the application of \cite[Propositions~2.2 and 5.1]{BelottodaSilvaFantiniPichon2022}.
	However, we appealed to different arguments because cusp singularities are not LNE in general.
	For instance, given positive integers $q$ and $r$ with $1/q+1/r<1/2$, the cusp singularity $T_{2qr}$, which is the hypersurface in $\C^3$ defined by the equation $x^2+y^q+z^r+\lambda xyz=0$ for some $\lambda\neq0$, is not LNE because its tangent cone at $0$, being defined by $x^2=0$, is not reduced (see \cite{SampaioFernandes2018} and \cite{DenkowskiTibar2019}, where the fact that LNE germs have reduced tangent cones is proved under varying degrees of generality).
	In any case, the final part of our argument, the one using the homotopy type of the foliations given by the Milnor fibers of a generic linear form, could not be simplified even for LNE singularities.
\end{remark}

\begin{remark}\label{remark:orientation_reversing_bilipschitz}
	Some of lens spaces and torus bundles over the circle admit orientation-reversing self-homeomorphisms (for a lens space $L(p,q)$, this happens precisely whenever $q^2\equiv-1$ modulo $p$).
	In this case, it is possible to construct explicitly an orientation-reversing bilipschitz homeomorphism from the corresponding Hirzebruch--Jung singularity to itself by using the bilipschitz model described in \cite[Section~11]{BirbrairNeumannPichon2014}.
	This is why Theorem~\ref{thm:bilipschitz type for complex surfaces} simply states that if an orientation-reversing homeomorphism between normal complex surface germs is bilipschitz, then another orientation-preserving homeomorphism must exist.
\end{remark}

\bibliographystyle{alpha}                              
\bibliography{biblio}

\begin{thebibliography}{BdSFNP22}

\bibitem[AB68]{AtiyahBott1968}
Michael~F. Atiyah and Raoul Bott.
\newblock A {L}efschetz fixed point formula for elliptic complexes. {II}.
  {A}pplications.
\newblock {\em Ann. of Math. (2)}, 88:451--491, 1968.

\bibitem[BCR98]{BochnakCosteRoy1998}
Jacek Bochnak, Michel Coste, and Marie-Fran\c{c}oise Roy.
\newblock {\em Real algebraic geometry}, volume~36 of {\em Ergebnisse der
  Mathematik und ihrer Grenzgebiete (3) [Results in Mathematics and Related
  Areas (3)]}.
\newblock Springer-Verlag, Berlin, 1998.
\newblock Translated from the 1987 French original, Revised by the authors.

\bibitem[BdSFNP22]{BelottodaSilvaFantiniNemethiPichon2022}
Andr\'{e} Belotto~da Silva, Lorenzo Fantini, Andr\'{a}s N\'{e}methi, and Anne
  Pichon.
\newblock Polar exploration of complex surface germs.
\newblock {\em Trans. Amer. Math. Soc.}, 375(9):6747--6767, 2022.

\bibitem[BdSFP22a]{BelottodaSilvaFantiniPichon2022a}
Andr\'{e} Belotto~da Silva, Lorenzo Fantini, and Anne Pichon.
\newblock Inner geometry of complex surfaces: a valuative approach.
\newblock {\em Geom. Topol.}, 26(1):163--219, 2022.

\bibitem[BdSFP22b]{BelottodaSilvaFantiniPichon2022}
Andr\'{e} Belotto~da Silva, Lorenzo Fantini, and Anne Pichon.
\newblock On {L}ipschitz normally embedded complex surface germs.
\newblock {\em Compos. Math.}, 158(3):623--653, 2022.

\bibitem[BNP14]{BirbrairNeumannPichon2014}
Lev Birbrair, Walter~D. Neumann, and Anne Pichon.
\newblock The thick-thin decomposition and the bilipschitz classification of
  normal surface singularities.
\newblock {\em Acta Math.}, 212(2):199--256, 2014.

\bibitem[Bri68]{Brieskorn1967/1968}
Egbert Brieskorn.
\newblock Rationale {S}ingularit\"aten komplexer {F}l\"achen.
\newblock {\em Invent. Math.}, 4:336--358, 1967/1968.

\bibitem[Che24]{Cherik2024}
Yenni Cherik.
\newblock Lipschitz geometry of complex surface germs via inner rates of
  primary ideals.
\newblock {\em To appeaer in Ann. Inst. Fourier, arXiv preprint
  arXiv:2407.14265}, 2024.

\bibitem[DT19]{DenkowskiTibar2019}
Maciej Denkowski and Mihai Tib\u{a}r.
\newblock Testing {L}ipschitz non-normally embedded complex spaces.
\newblock {\em Bull. Math. Soc. Sci. Math. Roumanie (N.S.)},
  62(110)(1):93--100, 2019.

\bibitem[Dur83]{Durfee1983a}
Alan~H. Durfee.
\newblock Neighborhoods of algebraic sets.
\newblock {\em Trans. Amer. Math. Soc.}, 276(2):517--530, 1983.

\bibitem[DW15]{DoigWehrli2015}
Margaret Doig and Stephan Wehrli.
\newblock A combinatorial proof of the homology cobordism classification of
  lens spaces.
\newblock {\em arXiv preprint arXiv:1505.06970}, 2015.

\bibitem[FP23]{FantiniPichon2023}
Lorenzo Fantini and Anne Pichon.
\newblock On {L}ipschitz normally embedded singularities.
\newblock In {\em Handbook of geometry and topology of singularities {IV}},
  pages 497--519. Springer, Cham, 2023.

\bibitem[Gre13]{Greene2013}
Joshua~Evan Greene.
\newblock Lattices, graphs, and {C}onway mutation.
\newblock {\em Invent. Math.}, 192(3):717--750, 2013.

\bibitem[Hat02]{Hatcher2002}
Allen Hatcher.
\newblock {\em Algebraic topology}.
\newblock Cambridge University Press, Cambridge, 2002.

\bibitem[Hir53]{Hirzebruch1953}
Friedrich Hirzebruch.
\newblock \"{U}ber vierdimensionale {R}iemannsche {F}l\"{a}chen mehrdeutiger
  analytischer {F}unktionen von zwei komplexen {V}er\"{a}nderlichen.
\newblock {\em Math. Ann.}, 126:1--22, 1953.

\bibitem[Hir73]{Hirzebruch1973}
Friedrich Hirzebruch.
\newblock Hilbert modular surfaces.
\newblock {\em Enseign. Math. (2)}, 19:183--281, 1973.

\bibitem[HNK71]{HirzebruchNeumannKoh1971}
Friedrich Hirzebruch, Walter~D. Neumann, and Sebastian~S. Koh.
\newblock {\em Differentiable manifolds and quadratic forms}.
\newblock Lecture Notes in Pure and Applied Mathematics, Vol. 4. Marcel Dekker,
  Inc., New York, 1971.
\newblock Appendix II by W. Scharlau.

\bibitem[Kin76]{King1976}
Henry~C. King.
\newblock Real analytic germs and their varieties at isolated singularities.
\newblock {\em Invent. Math.}, 37(3):193--199, 1976.

\bibitem[Lau73]{Laufer1973a}
Henry~B. Laufer.
\newblock Taut two-dimensional singularities.
\newblock {\em Math. Ann.}, 205:131--164, 1973.

\bibitem[Lau77]{Laufer1977}
Henry~B. Laufer.
\newblock On minimally elliptic singularities.
\newblock {\em Amer. J. Math.}, 99(6):1257--1295, 1977.

\bibitem[LMW01]{LeMaugendreWeber2001}
D\~{u}ng~Tr\'{a}ng L{\^{e}}, H\'{e}l\`ene Maugendre, and Claude Weber.
\newblock Geometry of critical loci.
\newblock {\em J. London Math. Soc. (2)}, 63(3):533--552, 2001.

\bibitem[LP05]{LuengoPichon2005}
Ignacio Luengo and Anne Pichon.
\newblock L\^{e}'s conjecture for cyclic covers.
\newblock In {\em Singularit\'{e}s {F}ranco-{J}aponaises}, volume~10 of {\em
  S\'{e}min. Congr.}, pages 163--190. Soc. Math. France, Paris, 2005.

\bibitem[Moi52]{Moise1952}
Edwin~E. Moise.
\newblock Affine structures in {$3$}-manifolds. {V}. {T}he triangulation
  theorem and {H}auptvermutung.
\newblock {\em Ann. of Math. (2)}, 56:96--114, 1952.

\bibitem[Neu81]{Neumann1981}
Walter~D. Neumann.
\newblock A calculus for plumbing applied to the topology of complex surface
  singularities and degenerating complex curves.
\newblock {\em Trans. Amer. Math. Soc.}, 268(2):299--344, 1981.

\bibitem[NP07]{NeumannPichon2007}
Walter~D. Neumann and Anne Pichon.
\newblock Complex analytic realization of links.
\newblock In {\em Intelligence of low dimensional topology 2006}, volume~40 of
  {\em Ser. Knots Everything}, pages 231--238. World Sci. Publ., Hackensack,
  NJ, 2007.

\bibitem[NPP20]{NeumannPedersenPichon2020b}
Walter~D. Neumann, Helge~M{\o}ller Pedersen, and Anne Pichon.
\newblock Minimal surface singularities are {L}ipschitz normally embedded.
\newblock {\em J. Lond. Math. Soc. (2)}, 101(2):641--658, 2020.

\bibitem[Orl72]{Orlik1972}
Peter Orlik.
\newblock {\em Seifert manifolds}.
\newblock Lecture Notes in Mathematics, Vol. 291. Springer-Verlag, Berlin-New
  York, 1972.

\bibitem[Pic01]{Pichon2001}
Anne Pichon.
\newblock Fibrations sur le cercle et surfaces complexes.
\newblock {\em Ann. Inst. Fourier (Grenoble)}, 51(2):337--374, 2001.

\bibitem[PP07]{Popescu-Pampu2007}
Patrick Popescu-Pampu.
\newblock The geometry of continued fractions and the topology of surface
  singularities.
\newblock In {\em Singularities in geometry and topology 2004}, volume~46 of
  {\em Adv. Stud. Pure Math.}, pages 119--195. Math. Soc. Japan, Tokyo, 2007.

\bibitem[Sae89]{Saeki1989}
Osamu Saeki.
\newblock Topological types of complex isolated hypersurface singularities.
\newblock {\em Kodai Math. J.}, 12(1):23--29, 1989.

\bibitem[SF18]{SampaioFernandes2018}
J~Edson Sampaio and Alexandre Fernandes.
\newblock {Tangent Cones of Lipschitz Normally Embedded Sets Are Lipschitz
  Normally Embedded. Appendix by Anne Pichon and Walter D. Neumann}.
\newblock {\em Int. Math. Res. Not. IMRN}, 01 2018.

\bibitem[Tro24]{Trotman2024}
David Trotman.
\newblock On semialgebraic homeomorphisms between semialgebraic germs.
\newblock {\em To appear in a volume dedicated to Bernard Teissier, edited by
  the European Mathematical Society Press, estimated publication date 2026.
  Preprint available at https://hal.science/hal-05143858v1}, 2024.

\bibitem[Tur88]{Turaev1988}
Vladimir~G. Turaev.
\newblock Towards the topological classification of geometric {$3$}-manifolds.
\newblock In {\em Topology and geometry---{R}ohlin {S}eminar}, volume 1346 of
  {\em Lecture Notes in Math.}, pages 291--323. Springer, Berlin, 1988.

\bibitem[Wal68]{Waldhausen1968}
Friedhelm Waldhausen.
\newblock On irreducible {$3$}-manifolds which are sufficiently large.
\newblock {\em Ann. of Math. (2)}, 87:56--88, 1968.

\bibitem[Web18]{Weber2018}
Claude Weber.
\newblock Lens spaces among 3-manifolds and quotient surface singularities.
\newblock {\em Rev. R. Acad. Cienc. Exactas F\'{\i}s. Nat. Ser. A Mat. RACSAM},
  112(3):893--914, 2018.

\end{thebibliography}

\vfill

\end{document}